\documentclass[11pt]{article}

\usepackage[margin=1in]{geometry}
\usepackage{amsmath,amssymb,amsthm}
\usepackage{mathrsfs}
\usepackage{hyperref}
\usepackage{enumitem}

\hypersetup{
  colorlinks=true,
  linkcolor=blue,
  citecolor=blue,
  urlcolor=blue
}

\newtheorem{theorem}{Theorem}
\newtheorem{lemma}{Lemma}
\newtheorem{remark}{Remark}

\newcommand{\R}{\mathbb{R}}

\title{Inequalities on a Class of Function Sets}
\author{Gangsong Leng\\
East China Normal University\\
\href{mailto: Gangsong Leng<lenggangsong@163.com>}{lenggangsong@163.com}
}
\date{}

\begin{document}
\maketitle

\begin{abstract}
We prove a functional extension of an exponential inequality originally proposed by Bin Zhao and proved by Xiaosheng Mou. The main result asserts that if
$\alpha_1\leq \cdots\leq \alpha_n$ and $\sum_{k=1}^n \alpha_k=0$, then
\[
  \sum_{k=1}^n \phi(k\alpha_k)\geq 0
\]
for every odd function $\phi$ that is increasing and convex on $[0,\infty)$. The proof is based on a truncated-sum comparison and the stop-loss characterization of the increasing convex order. As consequences, we recover the original exponential inequality and obtain polynomial and integral variants.
\end{abstract}


\section{Introduction}

The following theorem was first proposed by Bin Zhao in the ninth issue of the problem-solving column of Mathematics New Star Network. Xiaosheng Mou gave a proof of it; see \cite{Mou2019}.

\begin{theorem}\label{thm:original}
Let
\[
  \alpha_1\leq \alpha_2\leq \cdots\leq \alpha_n,
  \qquad \sum_{k=1}^n\alpha_k=0.
\]
Then
\[
  \sum_{k=1}^n e^{k\alpha_k}
  \geq
  \sum_{k=1}^n e^{-k\alpha_k}.
\]
\end{theorem}

Theorem \ref{thm:original} is equivalent to the following form: if
\[
  0<x_1\leq x_2\leq \cdots\leq x_n,
  \qquad x_1x_2\cdots x_n=1,
\]
then
\[
  x_1+x_2^2+\cdots+x_n^n
  \geq
  x_1^{-1}+x_2^{-2}+\cdots+x_n^{-n}.
\]
Indeed, setting $x_k=e^{\alpha_k}$ transforms the two statements into each other.

This note gives a functional generalization of Theorem \ref{thm:original}. Let $\mathscr S$ denote the set of all functions $\phi:\R\to\R$ satisfying the following two conditions:
\begin{enumerate}[label=\textup{(\arabic*)}]
  \item $\phi$ is odd;
  \item $\phi$ is increasing and convex on $[0,\infty)$.
\end{enumerate}
Since $\phi$ is odd, we automatically have $\phi(0)=0$. Clearly, $\mathscr S$ is closed under nonnegative linear combinations; hence it is a convex cone, not a linear space.

\section{The main result}

\begin{theorem}\label{thm:main}
Let
\[
  \alpha_1\leq \alpha_2\leq \cdots\leq \alpha_n,
  \qquad \sum_{j=1}^n\alpha_j=0.
\]
If
\[
  \phi=\sum_{i=1}^m c_i\phi_i,
  \qquad c_i\geq 0,
  \qquad \phi_i\in \mathscr S
  \quad (i=1,2,\ldots,m),
\]
then
\[
  \sum_{j=1}^n\phi(j\alpha_j)\geq 0.
\]
In particular, the same conclusion holds for every $\phi\in\mathscr S$.
\end{theorem}

\begin{proof}
It suffices to prove the conclusion for $\phi\in\mathscr S$; the stated form then follows by taking nonnegative linear combinations.

Suppose
\[
  \alpha_1\leq \cdots\leq \alpha_K<0\leq
  \alpha_{K+1}\leq \cdots\leq \alpha_n.
\]
If all $\alpha_k$ are zero, the conclusion is immediate. We now assume that such a $K$ exists.

We shall use two lemmas. The first lemma is the key point.

\begin{lemma}\label{lem:key}
Under the above assumptions, for every $t\geq 0$,
\[
  \sum_{k>K}(k\alpha_k-t)_+
  \geq
  \sum_{k\leq K}(-k\alpha_k-t)_+,
\]
where $u_+=\max\{u,0\}$.
\end{lemma}

\begin{proof}[Proof of Lemma \ref{lem:key}]
Put
\[
  A=\sum_{k>K}(k\alpha_k-t)_+,
  \qquad
  B=\sum_{k\leq K}(-k\alpha_k-t)_+.
\]
If $B=0$, the conclusion is obvious. Assume below that $B>0$. Since $B$ is the sum of $K$ nonnegative terms, there exists $l\leq K$ such that
\[
  -l\alpha_l-t\geq \frac{B}{K}.
\]
From $\alpha_1\leq \cdots\leq \alpha_l<0$, we obtain
\[
  -\sum_{k=1}^l\alpha_k-t
  \geq -l\alpha_l-t
  \geq \frac{B}{K},
\]
and hence
\[
  -\sum_{k=1}^l\alpha_k\geq \frac{B}{K}+t.
\]
Using the zero-sum condition, we get
\begin{equation}\label{eq:positive-sum}
  \sum_{k>K}\alpha_k
  =-\sum_{k\leq K}\alpha_k
  \geq -\sum_{k\leq l}\alpha_k
  \geq \frac{B}{K}+t.
\end{equation}

First suppose that $n\leq 3K+1$. Since $u_+\geq u$, we have
\[
  A\geq \sum_{k>K}(k\alpha_k-t)
  =\sum_{k>K}k\alpha_k-(n-K)t.
\]
Because $K+1,K+2,\ldots,n$ and $\alpha_{K+1},\ldots,\alpha_n$ are similarly ordered increasing sequences, the rearrangement inequality gives
\[
  \sum_{k>K}k\alpha_k
  \geq
  \frac{n+K+1}{2}\sum_{k>K}\alpha_k.
\]
Combining this with \eqref{eq:positive-sum}, we obtain
\[
  A\geq
  \frac{n+K+1}{2}\left(\frac{B}{K}+t\right)
  -(n-K)t
  =
  \frac{n+K+1}{2K}B
  +\frac{3K+1-n}{2}t.
\]
When $n\leq 3K+1$, the right-hand side is at least $B$. Thus $A\geq B$.

Now suppose that $n>3K+1$. Take
\[
  M=n-2K-1.
\]
Then $M>K$. We estimate $A$ using only the later tail:
\[
  A\geq \sum_{k>M}(k\alpha_k-t)
  =\sum_{k>M}k\alpha_k-(n-M)t.
\]
Again, by the rearrangement inequality,
\[
  \sum_{k>M}k\alpha_k
  \geq
  \frac{n+M+1}{2}\sum_{k>M}\alpha_k.
\]
Since $\alpha_{K+1}\leq \cdots\leq \alpha_n$, the average of the last $n-M$ terms is no smaller than the average of all $n-K$ positive-side terms. Hence
\[
  \sum_{k>M}\alpha_k
  \geq
  \frac{n-M}{n-K}\sum_{k>K}\alpha_k.
\]
Using \eqref{eq:positive-sum} again, we get
\[
  A\geq
  \frac{(n+M+1)(n-M)}{2(n-K)}
  \left(\frac{B}{K}+t\right)
  -(n-M)t.
\]
That is,
\[
  A\geq
  \frac{(n+M+1)(n-M)}{2K(n-K)}B
  +
  \frac{(M+2K+1-n)(n-M)}{2(n-K)}t.
\]
Substituting $M=n-2K-1$, the second term is exactly zero, and
\[
  \frac{(n+M+1)(n-M)}{2K(n-K)}
  =\frac{2K+1}{K}>1.
\]
Therefore $A\geq B$. The lemma is proved in both cases.
\end{proof}

The second lemma is the truncated-sum form of Karamata's theorem. In stochastic order theory, this form is also called the stop-loss characterization of the increasing convex order; see Shaked and Shanthikumar \cite[Chapter 4.A]{ShakedShanthikumar2007}.

\begin{lemma}[Truncated-sum form of Karamata's theorem]\label{lem:karamata}
Let $a_1,\ldots,a_r$ and $b_1,\ldots,b_s$ be nonnegative numbers. Suppose that for every $t\geq 0$,
\[
  \sum_{i=1}^r(a_i-t)_+
  \geq
  \sum_{j=1}^s(b_j-t)_+.
\]
If $F:[0,\infty)\to\R$ is increasing and convex, and $F(0)=0$, then
\[
  \sum_{i=1}^r F(a_i)
  \geq
  \sum_{j=1}^s F(b_j).
\]
\end{lemma}

\begin{remark}
The reason Lemma \ref{lem:karamata} implies inequalities for convex functions is that convex functions can be decomposed into superpositions of the truncated functions $(x-t)_+$. If $F$ is sufficiently smooth and $F(0)=0$, then one has a representation of the form
\[
  F(x)=F'_+(0)x+\\int_0^\infty (x-t)_+\,dF'_+(t).
\]
Since $F$ is convex, its right derivative $F'_+$ is increasing, and therefore $dF'_+(t)$ is a nonnegative measure. Thus, if the truncated-sum inequality holds for every $t\geq 0$, then multiplying by nonnegative weights and integrating gives the desired convex-function inequality. The general case follows by approximation with piecewise linear convex functions. This is precisely the truncated-sum form of Karamata's theorem, also known as the stop-loss characterization of the increasing convex order.
\end{remark}

We now return to the proof of Theorem \ref{thm:main}. Let
\[
  P_k=k\alpha_k \quad (k>K),
  \qquad
  Q_k=-k\alpha_k \quad (k\leq K).
\]
By Lemma \ref{lem:key}, for every $t\geq 0$,
\begin{equation}\label{eq:truncated-comparison}
  \sum_{k>K}(P_k-t)_+
  \geq
  \sum_{k\leq K}(Q_k-t)_+.
\end{equation}
By Lemma \ref{lem:karamata}, taking $F=\phi|_{[0,\infty)}$, we get
\[
  \sum_{k>K}\phi(k\alpha_k)
  \geq
  \sum_{k\leq K}\phi(-k\alpha_k).
\]
Since $\phi$ is odd, for $k\leq K$ we have
\[
  \phi(-k\alpha_k)=-\phi(k\alpha_k).
\]
Therefore
\[
  \sum_{k=1}^n\phi(k\alpha_k)\geq 0.
\]
Thus the assertion holds for each $\phi_i\in\mathscr S$. Multiplying by $c_i\geq 0$ and summing gives the theorem.
\end{proof}

Taking
\[
  \phi(t)=e^t-e^{-t}=2\sinh t,
\]
we have $\phi\in\mathscr S$, and Theorem \ref{thm:original} follows immediately from Theorem \ref{thm:main}.

As a direct application, if
\[
  P(t)=a_1t+a_3t^3+\cdots+a_{2r+1}t^{2r+1},
  \qquad a_{2j+1}\geq 0 \quad (j=0,1,\ldots,r),
\]
then $P\in\mathscr S$. Hence, under the assumptions of Theorem \ref{thm:main},
\[
  \sum_{k=1}^n P(k\alpha_k)\geq 0.
\]
In other words, every polynomial consisting of odd-degree terms with nonnegative coefficients yields an inequality of the same type.

\section{An integral form}

Finally, we point out that Theorem \ref{thm:main} has the following integral form.

\begin{theorem}\label{thm:integral}
Let $f$ be an increasing integrable function on $[0,1]$, and suppose that
\[
  \int_0^1 f(x)\,dx=0.
\]
If $\phi$ is odd and is increasing and convex on $[0,\infty)$, then
\[
  \int_0^1 \phi\bigl(xf(x)\bigr)\,dx\geq 0.
\]
\end{theorem}

\begin{proof}
First suppose that $f$ is continuous on $[0,1]$. For each positive integer $n$, set
\[
  \overline f_n=\frac1n\sum_{k=1}^n f\left(\frac{k}{n}\right),
\]
and
\[
  \alpha_k=
  \frac1n\left(
  f\left(\frac{k}{n}\right)-\overline f_n
  \right),
  \qquad k=1,2,\ldots,n.
\]
Since $f$ is increasing, we have
\[
  \alpha_1\leq \alpha_2\leq \cdots\leq \alpha_n,
  \qquad \sum_{k=1}^n\alpha_k=0.
\]
By Theorem \ref{thm:main},
\[
  \sum_{k=1}^n
  \phi\left(
  \frac{k}{n}
  \left(
  f\left(\frac{k}{n}\right)-\overline f_n
  \right)
  \right)
  \geq 0.
\]
Dividing both sides by $n$ and letting $n\to\infty$, and noting that
\[
  \overline f_n\to \int_0^1 f(x)\,dx=0,
\]
we obtain from the limit of Riemann sums that
\[
  \int_0^1\phi\bigl(xf(x)\bigr)\,dx\geq 0.
\]
For a general increasing integrable function, the conclusion follows by approximation with increasing step functions and then passing to the limit.
\end{proof}

\end{document}